\documentclass[12pt]{article}

\usepackage{amsmath}
\usepackage{amssymb}
\usepackage{amsthm}
\usepackage{color}
\usepackage{graphicx}
\usepackage{epstopdf}
\usepackage{subcaption}
\usepackage{psfrag}
\usepackage{float}
\usepackage{hyperref}
\usepackage{cite}
\usepackage{fullpage}

\allowdisplaybreaks
\newcommand{\bfc}{{\bf c}}
\newcommand{\bft}{{\bf t}}
\newcommand{\cT}{\mathcal{T}}
\newcommand{\cC}{\mathcal{C}}

\newcommand{\wc}{\widetilde{c}}

\newcommand{\wT}{\widetilde{T}}

\newcommand{\wbfc}{\widetilde{\bf c}}
\newcommand{\wbft}{\widetilde{\bf t}}
\usepackage{dsfont}

\newcommand\widebar[1]{\mathop{\overline{#1}}}

\newtheorem{remark}{Remark}
\newtheorem{theorem}{Theorem}
\newtheorem{lemma}{Lemma}
\theoremstyle{definition}
\newtheorem{definition}{Definition}

\begin{document}


\title{Optimal polymer slugs injection profiles}

\author{
F.~Bakharev,
A.~Enin, 
K.~Kalinin,
Yu.~Petrova, 
N.~Rastegaev,
S.~Tikhomirov\footnote{Corresponding author. E-mail: sergey.tikhomirov@gmail.com, Instituto de Matemática Pura e Aplicada, Estrada Dona Castorina, 110, Jardim Botânico  CEP 22460-320, Rio de Janeiro, RJ - Brasil, St. Petersburg State University, 7/9 Universitetskaya Emb., St Petersburg 199034, Russia}
}

\maketitle

\begin{abstract}The graded viscosity banks technology (tapering) for polymer flooding is studied for several different models of mixing zones behavior. Depending on the viscosity function the limiting polymer injection profile is rigorously derived for the transverse flow equilibrium model,
for the Koval model and for the Todd-Longstaff model. The potential
gain in the polymer quantity compared to the profile with a finite number of slugs is numerically estimated.
\end{abstract}




\section{Introduction}
Chemical flooding, i.e. injection of aqueous solutions of certain 
chemical components, is one of the most widely studied and implemented 
classes of enhanced oil recovery methods. Polymer flooding, a practice of adding polymer to the injected water to decrease its mobility, is a common example of such methods and is most often applied for one of the following purposes. Firstly, to stabilize the displacement when the oil-water shock front mobility ratio is greater than one and thus there is a high risk of 
viscous fingering instability occurring (see \cite{OttBerg2013,Hagoort1974, Blunt1994, BaCaEnMaPeTiYa2020}). Secondly, in combination with other chemicals to increase the shock front water saturation compared to the ordinary waterflooding, thus improving the oil displacement and delaying the water breakthrough to a producing well. 
Thirdly, to improve the sweep efficiency in heterogeneous reservoirs. The main common disadvantage of chemical flooding methods is the high cost of the utilized chemical compounds. In order to reduce the costs it is preferable to inject limited volumes of chemical solutions (slugs). Generally in such cases the main chemical slug is driven further by an additional polymer slug and then a water postflush. 

However, some factors may potentially impede or damage 
the polymer slug. In particular, if the transition to waterdrive is made abruptly, miscible viscous fingering occurs at the trailing edge of the polymer slug, thereby forming a so-called mixing zone (see Fig.~\ref{Fig-01}). If these fingers (the leading edge of the mixing zone) reach the main slug,
they will affect its composition 
and reduce its benefits. Note that viscous fingering in this case usually occurs despite the stabilizing effect of diffusion, see e.g. \cite{ChuokeMeurs1959, Outmans1962, Perrine1961-I, Perrine1961-II,Claridge}. We mention here that another crucial factor which negatively affects the polymer slug is adsorption, but in the current work we are not going to discuss it. 

\begin{figure}
    \centering
    \includegraphics[width=0.61\textwidth]{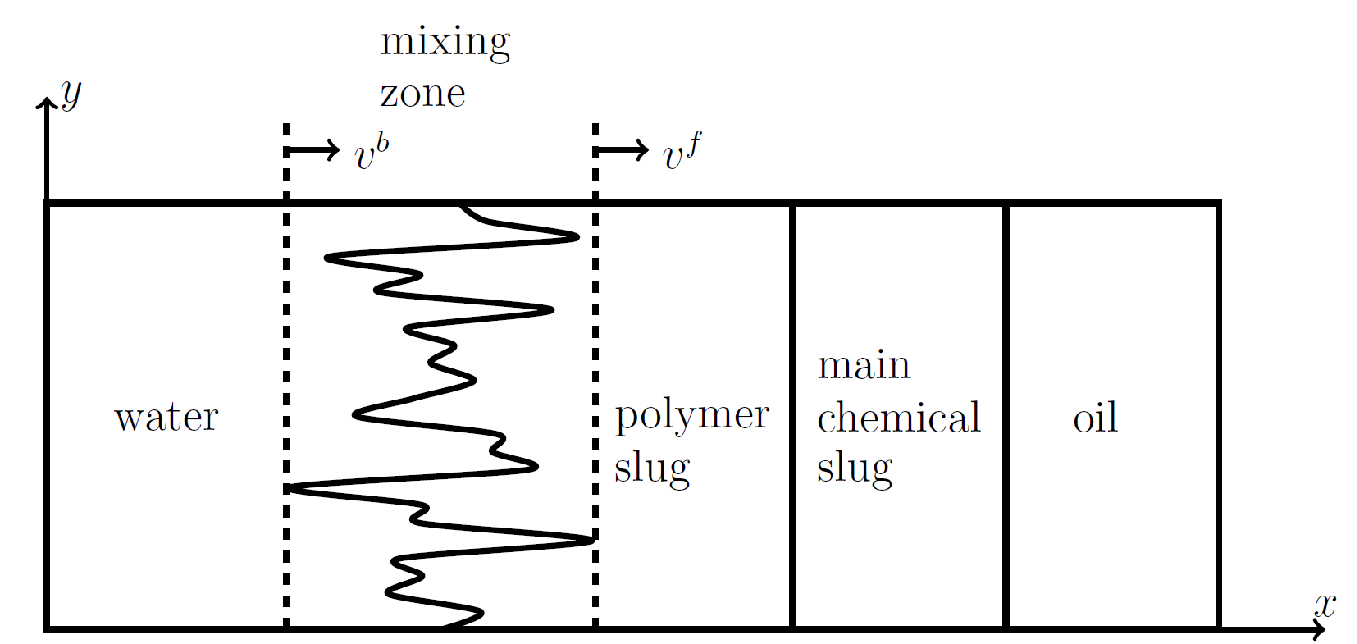}
    \caption{Example of an injection scheme in chemical flooding. On the trailing edge of the polymer slug instabilities occur, forming a so-called mixing zone of viscous fingering. In this paper we assume constant velocity rates of leading and trailing edges of the mixing zone, $v^f$ and $v^b$ respectively}
    \label{Fig-01}
\end{figure}

Many numerical simulations and laboratory experiments, e.g. \cite{Nijjer2018, Linear} and references therein, show that in the case 
of a high P\'eclet number the mixing zone spreads linearly in time and its leading and trailing edges move with asymptotically constant velocities $v^f$ and $v^b$ respectively. There are several empirical  models describing the behavior of the averaged concentration across the fingers, and as a consequence giving the empirical estimates for the velocities $v^f$ and $v^b$ (Koval model --- \cite{Koval}, naive Koval model --- \cite{Booth}, Todd-Longstaff model --- \cite{TL}, etc.). Another approach consists in writing pessimistic theoretical bounds for the velocities under some additional assumptions, e.g.~transverse flow equilibrium (TFE) model (see \cite{Otto2005,Otto2006} for gravity-driven fingers, \cite{Yortsos} for linear flood and \cite{Linear} for numerical elaboration). The main idea of these approaches is to reduce a two-dimensional problem to a one-dimensional one, which greatly simplifies the analysis of corresponding equations.

In this paper we focus on the graded viscosity bank (GVB) technology initially suggested in~\cite{Claridge}. It consists in sequential injection of several polymer slugs with a gradual decrease in viscosity (concentration), which helps to reduce the viscous fingering instabilities and allows to save polymer compared to one polymer slug. In this case several mixing zones appear and it is important to prevent their overlapping. In present paper we introduce generalized Koval model, which includes Koval, naive Koval and Todd-Logstaff models. The main result of the paper is the description of optimal injection schemes for TFE and generalized Koval models. 
The choice of model is out of the scope of the present paper, see \cite{Linear} for the discussion. Even though due to engineering and organizational reasons it is easier in practice to implement a small number of slugs, it is important to specify the maximum theoretically gain in polymer quantity for the GVB-method in order to make a decision on the number of slugs to be used.

One notable conclusion of our analysis is that the limiting optimal injection schemes for all considered models have the same simple structure. After some natural rescaling it starts with the maximal polymer concentration $c_{\max}$ at the beginning of the polymer injection and decays as the inverse to the function (see Section~\ref{sec:problem-statement} for formal definition)
$$
\cT(c)=1-\left(\frac{\mu(c)}{\mu(c_{\max})}\right)^{\beta}
$$
with some parameter $\beta$ specific to each model. Note that $\beta=0.44$ for the Koval model was heuristically predicted in  \cite{Claridge}, but was not rigorously proved. In this paper we get that result as a consequence of a general scheme.

The paper has the following structure. In Section~\ref{sec:problem-statement} we formulate the mathematical description of the graded viscosity banks technology for a finite number of slugs and the main result of the article (Theorem \ref{02-main_theorem}), that describes the limiting optimal injection profile in GVB technology as number of slugs tends to infinity. Section~\ref{sec:proof} contains the proof of Theorem \ref{02-main_theorem}. Section~\ref{sec:numerics} describes the numerical results for a finite number of slugs and some qualitative information about the convergence to limiting profiles.

\section{Problem statement and formulation of the main theorem}
\label{sec:problem-statement}
We assume the flow to be single-phase, which is the case, for example, for the postflush polymer slug in the surfactant-polymer flooding. One-phase miscible flow in porous media can be described by a system of equations called Peaceman model (see \cite{Peaceman1962}) which couples an equation for the transport of solvent with incompressibility condition and Darcy’s law for fluid velocity:
\begin{align}\label{01-Peaceman_model}
    &\phi\frac{\partial c}{\partial t} + \mathrm{div}(q\cdot c)=D\Delta c,
    \\
    &\label{02-Peaceman_model}
    q=-\frac{k}{\mu(c)}\nabla p=-m(c)\nabla p,\qquad \mathrm{div}(q)=0,
\end{align}
where $c$ is the concentration of the polymer, $\phi$ is the porosity, $q$ is the velocity, $\mu=\mu(c)$ is the viscosity function which is assumed to be positive and strictly increasing, $p$ is the pressure, $D$ is the diffusion coefficient and $k$ is the permeability. Also $m=k/\mu$ is the mobility function.
\begin{figure}[!h]
\begin{center}
\includegraphics[width=0.8\textwidth]{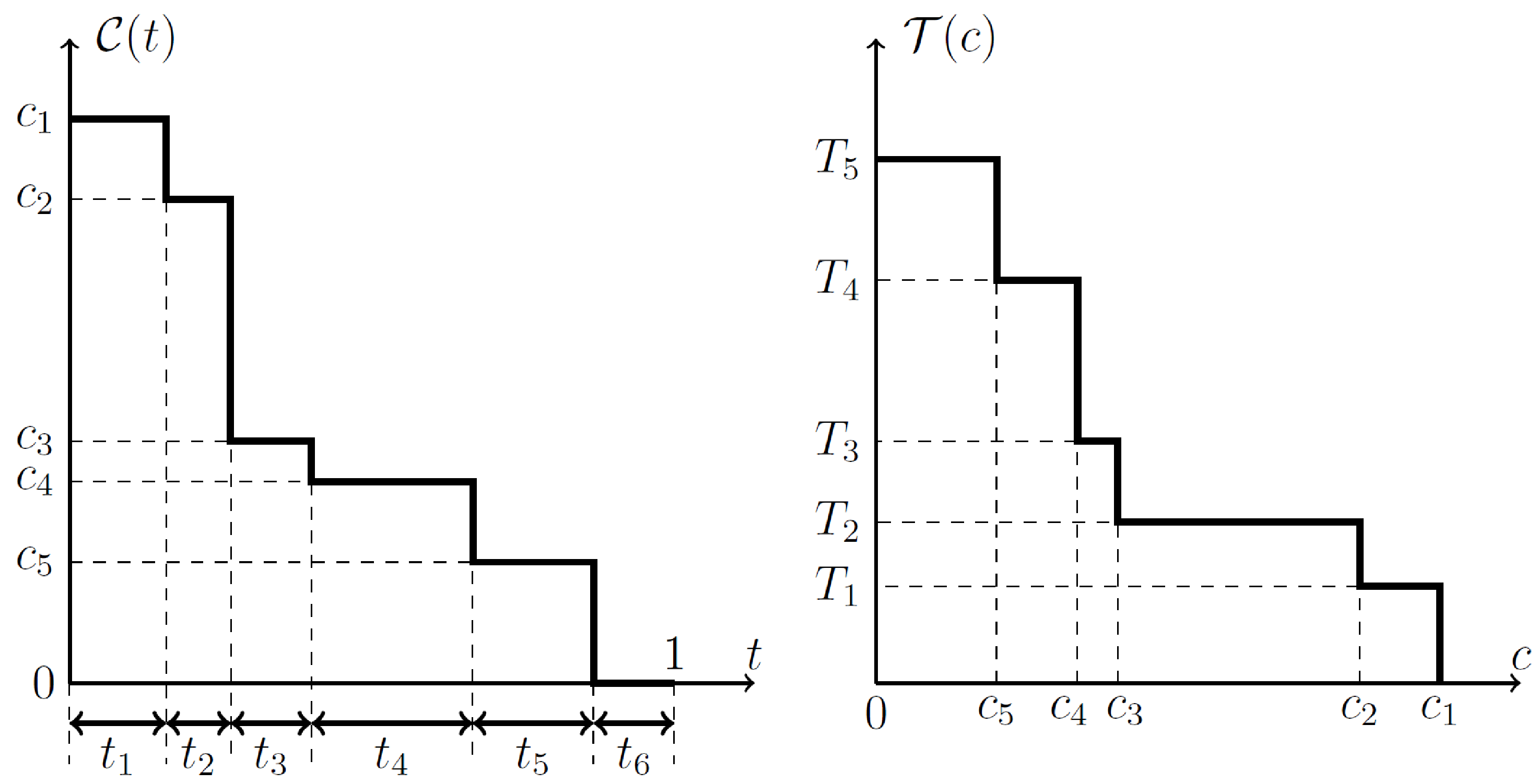}
\caption{An example of boundary condition $\cC(t)$ in GVB (left) and corresponding injection profile $\cT(c)$ (right)}
\label{fig:inj-profile}
\end{center}
\end{figure}

We consider \eqref{01-Peaceman_model}-\eqref{02-Peaceman_model} in a rectangular domain with length $L$ and width $W$ ($x \in [0, L]$, $y \in [0, W]$), and assume that $Q$ is a linear injection velocity at the injection well, i.e. $q(0,y,t)=(Q,0)$. Under a proper scaling of variables we can assume $L=1$, $Q=1$. We set $c=0$ as an initial condition, meaning there is no polymer in the reservoir before we start the injection. The boundary condition at the injection well for GVB technology is given by a piecewise constant function $\cC=\cC(t)=c(0,y,t)$ (see Fig.~\ref{fig:inj-profile}). The values of the injected concentrations are described by a partition $\bfc=(c_1,\ldots,c_{n+1})$ of the interval $[c_{\min},c_{\max}]=[c_{n+1}, c_1]$, i.e. $c_{\max}=c_1 > c_2 > \ldots >c_n > c_{n + 1}=c_{\min}$. The corresponding injection times are given by the vector $\bft=(t_1,\ldots, t_{n+1})$ with non-negative coordinates and $t_1>0$. Usually, $c_{\min}=c_{n+1}=0$ as we inject water as the final slug. 

For the convenience in formulations below we will use the following notation: for $a, b \in \mathds{R}, b> a$ the indicator function $\mathds{1}_{[a, b]}:\mathds{R} \to \mathds{R}$ is defined by the relation
$$
\mathds{1}_{[a, b]}(x) = 
\begin{cases}
1, & x \in [a, b]\\
0, & x \notin [a, b].
\end{cases}
$$
The function $\cC$ can be written as a linear combination of the indicator functions 
$$
\cC(t):=\sum_{j=1}^{n+1} c_j\mathds{1}_{[T_{j-1},T_j]}(t),
$$
where $T_0=0$ and $T_j=t_1+\ldots+t_j$ is the starting time for the injection of the concentration $c_{j+1}$, $T_{n+1}=L/Q=1$. It is convenient to consider an inverse function $\cT=\cT(c)$ defined by the formula
\begin{equation}
\label{02-injection_profile}
    \cT(c) := \sum_{j = 1}^{n} T_j\cdot\mathds{1}_{[c_{j + 1}, c_{j}]}(c),
\end{equation}
which we will call an {\it injection profile} (see Fig. \ref{fig:inj-profile}). 

The total amount $V$ of injected polymer may be calculated as
\begin{equation}
\label{02-volume_function}
        V(\bfc,\bft) := \sum\limits_{k=1}^{n+1} c_k t_k=\sum\limits_{j=1}^{n} (c_j-c_{j+1}) T_j + c_{n+1} T_{n+1}=\int\limits_0^{c_1}\cT(c)\,dc.
\end{equation}

As we mentioned above the mixing zones appear at each common boundary of two slugs (see Fig. \ref{fig:gvb-2dim}). We say there is a breakthrough if two mixing zones overlap each other. Let $v^f_j$ and $v^b_j$ be a leading and trailing edge velocities of the $j$-th mixing zone. We are interested in configurations in which a breakthrough does not occur before the polymer reaches the production well. Thus, the no-breakthrough conditions are $v_0^b=1$ and 
\begin{align}
\label{02-break_ineq_3}
     v_j^b\cdot(1-T_j)&\geqslant v_{j+1}^f\cdot (1-T_{j+1}), 
\end{align}
for $j\in\{0,\ldots,n-1\}$. 
\begin{remark}
    Conditions~\eqref{02-break_ineq_3} provide a no-breakthrough between each two consecutive slugs, although from a practical point of view it is sufficient to have a no-breakthrough condition of the first slug only. However, rigorous formulation of such a problem faces the lack of numerical and experimental investigation on what happens when two mixing zones overlap. It is not clear whether the dynamics of the leading and trailing edges of new mixing zones remain linear. Preliminary numerical experiments show that this dynamics obeys different laws. This is out of the scope of the present paper.
\end{remark}
\begin{figure}[ht]
    \centering
    \includegraphics[width=0.8\textwidth]{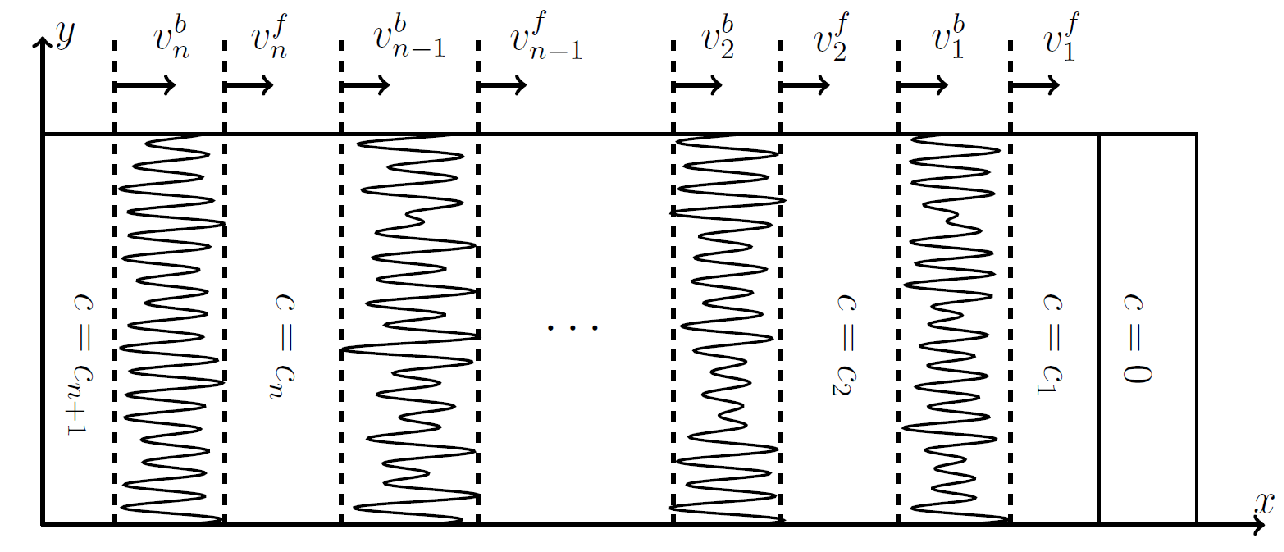}
    \caption{Two-dimensional picture of the viscous fingers growth in Peaceman model}
    \label{fig:gvb-2dim}
\end{figure}

\begin{definition}
\label{02-configuration}
    A pair $(\bfc, \bft)$ is called an $n$-configuration if $\bfc$ is the partition of $[c_1, c_{n + 1}]$ and $\bft$ satisfies the no-breakthrough conditions \eqref{02-break_ineq_3}.
\end{definition}

\begin{definition}
    An $n$-configuration $(\bfc, \bft)$ is called an optimal $n$-configuration if it minimizes the volume function $V$ among all $n$-configurations.
\end{definition}

\begin{remark}
\label{rmk1}
It is easy to see that for an optimal $n$-configuration the inequalities in~\eqref{02-break_ineq_3} could be replaced by equalities. Then the no-breakthrough conditions become a linear system of $n$ equations with $n$ variables and can be solved explicitly:
\begin{align}
\label{Tj-general}
    T_j=1-\prod\limits_{k=1}^j\frac{v_{k-1}^b}{v_k^f}, \quad j = 1,\ldots,n.
\end{align}
Note that the velocity functions $v^f_j$ and $v^b_j$ generally depend on the partition~$\bfc$. Concrete form of these dependencies is determined by the choice of the mixing zone growth model. So for an optimal $n$-configuration the volume function $V=V(\bfc,\bft(\bfc))$ is a function of $\bfc$ only. For simplicity we denote it with the same letter $V(\bfc)$.
\end{remark}

\begin{definition}
    Let $(\bfc, \bft)$ be an optimal $n$-configuration. Then $\cT$, defined in  \eqref{02-injection_profile}, is called an optimal injection profile.
\end{definition} 
\begin{definition}
    Let $\cT^n$ be an optimal injection profile for $n$ slugs. If there exists the pointwise limit $\cT^\infty$ of $\cT^n$ as $n \to \infty$ it is called an optimal limiting injection profile.
\end{definition}

Let us consider specific models used to estimate the growth of the mixing zone. 
\subsection*{\textbf{TFE model}} Rigorous pessimistic estimates for the velocities of the mixing zone ends were made under the assumption of transverse flow equilibrium condition (TFE) by \cite{Otto2006, Yortsos,wooding1969}. According to this model, the velocities of front end and the back end of the $j$-th mixing zone can be estimated by:
    \begin{align}\label{03-Otto}
        v^f_j \leqslant\frac{\widebar{m}(c_j,c_{j+1})}{m(c_j)} , 
        \qquad v^b_j \geqslant \frac{\widebar{m}(c_j,c_{j+1})}{m(c_{j + 1})},
    \end{align}
    where     
    \begin{align*}
        \widebar{m}(a,b)=\frac{1}{b-a}\int\limits_a^b m(c)\,dc
    \end{align*}
    is the mean value of the function $m$ on the interval $(a,b)$. In what follows for TFE model we will consider the velocities $v_j^f$ and $v_j^b$ to be equal to their pessimistic estimates (equalities in \eqref{03-Otto}). We will assume the following properties of $m$:
   \begin{align}
   \label{TFE-1}
        m\in C[c_{\min},c_{\max}],\; \text{strictly positive and decreasing.}
   \end{align} 
    
    \subsection*{\textbf{Generalized Koval model}}
    It is natural to conjecture that velocities of the leading and the trailing edge of $j$-mixing zone depend on the ratio of viscosities of the  concentrations $c_j$ and $c_{j+1}$, that is for any partition $\bfc=(c_1,\ldots,c_{n+1})$ we have
    \begin{align}
    \label{03-genKoval}
    &v_j^f=\frac{1}{f\left(\frac{m(c_{j+1})}{m(c_{j})}\right)},&& v_j^b=f\left(\frac{m(c_{j+1})}{m(c_{j})}\right),&& j=1,\ldots,n.
    \end{align}
    We will assume the following properties of $f$ and $m$:
    \vspace{3pt}
    \begin{enumerate}
        \item[(A)] $f\in C^1[1,M]$, where $M=\frac{m(c_{\min})}{m(c_{\max})}$, $m\in C^1[c_{\min},c_{\max}]$;
        \item[(B)] $f(1)=1$, $f>0$ and $f$ is decreasing;
        \item[(C)] $f(ab)\leqslant f(a)f(b)$ for any $a,b\in[1,M]$ such that $ab\leqslant M$;
        \item[(D)] $f\left(\frac{m(c')}{m(c)}\right)$ is a convex function of $c$ for any $c'\in[c_{\min}, c_{\max}]$, $c<c'$.
    \end{enumerate}
    The famous empirical models 
    (\cite{Koval},~\cite{TL}) are  of type~\eqref{03-genKoval} with~$f$:
    \begin{align*}
        &f_{K}(x)=(\alpha\cdot x^{1/4}+(1-\alpha))^{-4},
        && f_{TL}(x)=x^{-\omega}.
    \end{align*}
    Hereafter, the subscript ``K'' stands for the Koval model, while the subscript ``TL''  stands for the Todd-Longstaff model. In applications usually parameter  $\alpha\in(0,1)$ is taken to be equal to $0.22$; the miscibility parameter $\omega \in (0, 1]$ usually selected according to the data and sometimes is taken to be equal to $2/3$. In~\cite{Booth} there was an attempt to rigorously derive Koval model, which gave rise to naive Koval model, corresponding to TL model with $\omega=1$. It is easy to see that both Koval and TL models satisfy conditions (A)--(C). Notice that condition (D) is a restriction on mobility functions $m$ that we consider.
  
    
The main purpose of this article is to find the shape of the optimal limiting polymer injection profile. For TFE and generalized Koval models we give the answer in Theorem~\ref{02-main_theorem}. 

\begin{theorem}\label{02-main_theorem}
Under the assumption \eqref{TFE-1} for TFE model and (A)--(D) for generalized Koval model, the optimal limiting injection profile $\cT^\infty$ exists and has the form
        \begin{equation*}
             \cT^\infty(c)= 1-\left( \frac{m(c_1)}{m(c)}\right)^{\beta} = 1- \left(\frac{\mu(c)}{\mu(c_1)}\right)^{\beta}\,
        \end{equation*}
with $\beta=1$ in the case of TFE and $\beta=-2f'(1)$ in the case of generalized Koval model.
\begin{remark}
    In particular, for the Koval model we have $\beta=2\alpha$, which is in agreement with~\cite{Claridge} for $\alpha=0.22$. For TL model we obtain $\beta=2\omega$.
\end{remark}
\end{theorem}


\section{Proof of the main result}
\label{sec:proof}
Let $\bfc$ be a partition of $[c_{\min}, c_{\max}]$. 
In what follows we will need the notion of a rank of the partition $\bfc$:
$$
    \delta(\bfc):=\max\limits_{1\leq j\leq n} \lvert c_j-c_{j+1} \rvert.
$$
The proof of Theorem~\ref{02-main_theorem} consists of two parts: the existence of the optimal limiting injection profile and finding it's exact form. Using Lemmas~\ref{lm2}--\ref{lm3} we show the existence of the limiting profile and this is the most technical part of the proof.

Lemma~\ref{lm2} shows that for any optimal $n$-configuration there exists an $(n+1)$-configuration with less amount of polymer volume, so under the assumptions above it is always profitable to increase number of slugs. Note that it is not always the case (see Rem.~\ref{rm-m-nonconvex}).
\begin{lemma}
\label{lm2}
    Let $(\bfc, \bft)$ be an optimal $n$-configuration with volume $V_n$. 
    Then there exists an $(n+1)$-configuration $(\wbfc, \wbft)$ with volume $V_{n+1}$ and a function $\varphi$ such that
    \begin{align*}
        V_{n}(\bfc)-\varphi(\delta(\bfc))>V_{n+1}(\wbfc,\wbft). 
    \end{align*}
    Here function $\varphi: (0,c_{\max}-c_{\min}]\to\mathbb{R}_+$ is continuous, non-decreasing and strictly positive.
\end{lemma}
\begin{proof}
Fix an index $k\in\{1,\ldots,n\}$ such that $k:=\mathrm{argmax} |c_{k}-c_{k+1}|$ and consider a new vector of concentrations $\widetilde{\bfc}$:
\begin{align*}
    \bfc&=(c_1,\ldots,c_k,c_{k+1},\ldots,c_{n+1}),\\
    \wbfc:&=(c_1,\ldots,c_k,\tilde{c},c_{k+1},\ldots,c_{n+1}),
\end{align*}
for some $\wc\in[c_{k+1},c_{k}]$, which we will define later. Define $\wbft$ such that relation~\eqref{Tj-general} holds. For a configuration $(\wbfc, \wbft)$ we will use the same system of notations from formulas \eqref{02-volume_function}, \eqref{03-Otto}, \eqref{03-genKoval} by equipping the corresponding symbols with a tilde. Then
\begin{align}
\label{lm-vol-diff}
    V(\bfc)-V(\wbfc)&=\sum\limits_{i=1}^{k-1}(c_{i}-c_{i+1})(T_i-\wT_i) + \sum\limits_{i=k+1}^{n+1}(c_{i}-c_{i+1})(T_i-\wT_{i+1})+
    \\
    \nonumber
    &+(c_k-c_{k+1})T_k-((c_k-\wc)\wT_k+(\wc-c_{k+1})\wT_{k+1}).
\end{align}
We claim that for both TFE and generalized Koval models the following holds:
\begin{enumerate}
    \item[(i)] $T_i=\wT_i$ for $i\leqslant k-1$;
    \item[(ii)] $T_i\geqslant \wT_{i+1}$ for $i\geqslant k+1$;
    \item[(iii)] there exists $\tilde{c}$ and a continuous strictly positive function $\varphi$ such that
    \begin{align*}
        \Delta_k(\wbfc):=(c_k-c_{k+1})T_k-((c_k-\wc)\wT_k+(\wc-c_{k+1})\wT_{k+1})\geqslant \varphi(\delta(\bfc)).
    \end{align*}
\end{enumerate}
It is easy to see that Lemma~\ref{lm2} follows from (i)--(iii) and relation~\eqref{lm-vol-diff}.

Consider TFE model. Formula~\eqref{Tj-general} for $T_j$ (for any $j=2,\ldots,n$) simplifies significantly due to telescopic cancellations
\begin{align}
\label{Tj-TFE}
    T_j=1-\frac{m(c_1)}{\widebar{m}(c_j,c_{j+1})}.
\end{align} 
Using formula~\eqref{Tj-TFE}, we get that  $T_i=\wT_i$ for $i\leqslant k-1$, $i\geqslant k+1$ and (i)--(ii) are true. 

Let us prove (iii). Notice that for TFE model $\Delta_k(\wbfc)=\Delta_k(c_k,\wc,c_{k+1})$. Introducing $\lambda =(c_k-\wc)/(c_k-c_{k+1})\in [0,1]$ and rewriting $\Delta_k$, we get
\begin{align*}
    \Delta_k(c_k,\wc,c_{k+1})=(c_k-c_{k+1})(T_k-\lambda \wT_k - (1-\lambda)\wT_{k+1}).
\end{align*}
First, let us show that $\Delta_k>0$ for any $\wc\in(c_{k+1},c_k)$, or equivalently for $\lambda\in(0,1)$. By formula~\eqref{Tj-TFE} the inequality $\Delta_k>0$ is equivalent to
$$
\lambda \frac{1}{\widebar{m} (c_k, \wc)} +(1-\lambda) \frac{1}{\widebar{m}(\wc, c_{k+1})} > \frac{1}{ \widebar{m}(c_k,c_{k+1})},
$$
and can be rewritten as 
$$
\frac{\lambda}{A}+\frac{1-\lambda}{B}>\frac{1}{\lambda A+(1-\lambda)B}
$$
with $A=\widebar{m} (c_k, \wc)$ and $B=\widebar{m}(\wc, c_{k+1})$  and follows from the convexity of $1/x$ for $x>0$. Note that $\Delta_k=0$ when $\wc=c_k$ or $\wc=c_{k+1}$, or equivalently $\lambda=0$ or $\lambda=1$.

Secondly, the existence of $\varphi$ follows from the compactness argument. Define
\begin{align}
\label{lm-varphi}
    \varphi(d):=\min\limits_{\lvert c_k-c_{k+1} \rvert \geqslant d} \left(\max\limits_{\wc\in[c_{k+1},c_k]}\Delta_k(c_k,\wc,c_{k+1})\right).
\end{align}
Since $\Delta_k$ is a continuous function defined on a compact set, we obtain $\varphi$ to be continuous and strictly positive. By definition,  $\varphi$ is non-decreasing. Taking  $\wc:=\mathrm{argmax}\, \Delta_k(c_k,\wc,c_{k+1})$ for fixed values of $c_k, c_{k+1}\in [c_{\min}, c_{\max}]$, and $\varphi$ from formula~\eqref{lm-varphi}, we get $\Delta_k\geqslant \varphi(c_k-c_{k+1})=\varphi(\delta(\bfc))$, where the last inequality is true by the choice of $k$. This completes the proof of Lemma~\ref{lm2} for TFE model.

Consider generalized Koval model. 
From~\eqref{Tj-general} we obtain
\begin{align}
\label{Tj-Koval}
    T_j&=1-f\left(\frac{m(c_{j+1})}{m(c_j)}\right)\cdot g_j(\bfc);&&
    g_j(\bfc)=\prod\limits_{i=1}^{j-1}f^2\left(\frac{m(c_{i+1})}{m(c_i)}\right).
\end{align}
It follows from~\eqref{Tj-Koval}, $T_i=\wT_i$ for $i\leqslant k-1$. Inequality~(ii) for any $i\geqslant k+1$ is equivalent to
\begin{align*}
    f^2\left(\frac{m(c_{k+1})}{m(c_k)}\right)\leqslant f^2\left(\frac{m(\wc)}{m(c_k)}\right)\cdot f^2\left(\frac{m(c_{k+1})}{m(\wc)}\right),
\end{align*}
and is true due to assumption~(C) on function $f$. 

Let us prove (iii). Introducing $\lambda =(c_k-\wc)/(c_k-c_{k+1})\in [0,1]$ and using~\eqref{Tj-Koval}, we can rewrite $\Delta_k$ in the following form
\begin{align*}
    \Delta_k(\wbfc)=g_k(\bfc)\cdot(c_k-c_{k+1})\cdot\widetilde{\Delta}_k(c_k,\wc,c_{k+1}),
\end{align*}
where $\widetilde{\Delta}_k$ is defined as
\begin{align*}
    \widetilde{\Delta}_k:= \lambda f\left(\frac{m(\wc)}{m(c_k)}\right)+(1-\lambda)f^2\left(\frac{m(\wc)}{m(c_k)}\right)f\left(\frac{m(c_{k+1})}{m(\wc)}\right)-f\left(\frac{m(c_{k+1})}{m(c_k)}\right).
\end{align*}
First, let us show that $\Delta_k>0$ for any $\wc\in(c_{k+1},c_k)$, or equivalently for $\lambda\in(0,1)$. It is clear that $g_k>0$ and $(c_k-c_{k+1})>0$. 
Note that
\begin{align*}
    \cfrac{f\left(\frac{m(c_{k+1})}{m(c_k)}\right)}{f\left(\frac{m(\wc)}{m(c_k)}\right)}
    &\stackrel{\text{(C)}}{\leqslant}
    f\left(\frac{m(c_{k+1})}{m(\wc)}\right)
    \stackrel{\text{(D)}}{<}
    \lambda f\left(\frac{m(c_{k+1})}{m(c_{k+1})}\right) +(1-\lambda)f\left(\frac{m(c_{k+1})}{m(c_k)}\right)\stackrel{\text{(B)}}{=}
    \\
    &\stackrel{\text{(B)}}{=}
    \lambda  +(1-\lambda)f\left(\frac{m(c_{k+1})}{m(c_k)}\right)\stackrel{\text{(C)}}{\leqslant}
    \lambda +(1-\lambda)f\left(\frac{m(\wc)}{m(c_k)}\right)f\left(\frac{m(c_{k+1})}{m(\wc)}\right)
\end{align*}
and hence $\widetilde{\Delta}_k>0$.

Secondly, by the same compactness argument as for TFE model, we can show that there exists $\tilde{c}\in(c_{k+1},c_k)$ and a function $\widetilde{\varphi}:\,(0,c_{\max}-c_{\min}]\to\mathbb{R}_+$ such that $\widetilde{\Delta}_k\geqslant \widetilde{\varphi}(c_{k}-c_{k+1})$. Note that $\widetilde{\varphi}$ is continuous, non-decreasing and strictly positive.

Thirdly, let us prove that $g_{k}(\bfc)$ is separated from zero. Consider a function $z(x,y)=\ln \left(f\left(\frac{m(x)}{m(y)}\right)\right)$ for $x,y\in[c_{\min},c_{\max}]$ and $x\leqslant y$. By condition (A), we see that $z$ is a Lipschitz function. Thus there exists a positive constant $K>0$ such that
\begin{align*}
    \lvert z(x,y)-z(x,x)\rvert=\lvert z(x,y) \rvert \leqslant K \lvert x-y \rvert.
\end{align*}
This implies $g_k\geqslant e^{-2K}$ due to inequalities
\begin{align*}
    \lvert \ln \left(g_{k}(\bfc)\right) \rvert
    \leqslant 2\sum\limits_{i=1}^{k-1} \left| \ln \left(f\left(\frac{m(c_{i+1})}{m(c_i)}\right)\right) \right|
    \leqslant 2\sum\limits_{i=1}^{k-1}\lvert c_{j+1}-c_j\rvert K\leqslant 2K.
\end{align*}
Finally, we obtain
\begin{align*}
    \Delta_k(\wbfc)\geqslant e^{-2K}(c_k-c_{k+1}) \widetilde{\varphi}(c_{k}-c_{k+1})=:\varphi(c_{k}-c_{k+1})=\varphi(\delta(\bfc)),
\end{align*}
where the last inequality is true by the choice of $k$. This completes the proof of Lemma~\ref{lm2} for generalized Koval model.
\end{proof}

\begin{remark}
\label{rm-m-nonconvex}
For TL model in condition (C) equality  $f(ab)=f(a)f(b)$ holds. So  $m^{\omega}$ is convex if and only if for any 1-configuration $\bfc=(c_1,c_2)$ no matter what intermediate concentration $\wc\in(c_1,c_2)$ you add, there will be gain in polymer volume. So the convexity condition (D) provides the possibility of passing from $n$-configuration to $(n+1)$-configuration and as we will see  guarantees the existence of a continuous limiting injection profile.

Consider an example where condition (D) on $m^\omega$ is not satisfied.  Take $m^\omega (c)=1/(1+c^3)$, which corresponds to $\mu(c)=(1+c^3)^{3/2}$ and $\omega=2/3$. Note that the graph of $m^\omega$ lies above the segment, connecting points $(c_1,m^\omega(c_1))$ and $(c_2,m^\omega(c_2))$, so following the proof of Lemma~\ref{lm2} we get that any $2$-configuration is less optimal than $1$-configuration. Dependence of the volume function on $\wc$ is given in Fig.~\ref{tl_countrexample}. In this case we won't get any limiting profile as we can't even replace a 1-configuration with a better 2-configuration.
    
    \begin{figure}[!h]
    \centering
    \includegraphics[scale = 0.36]{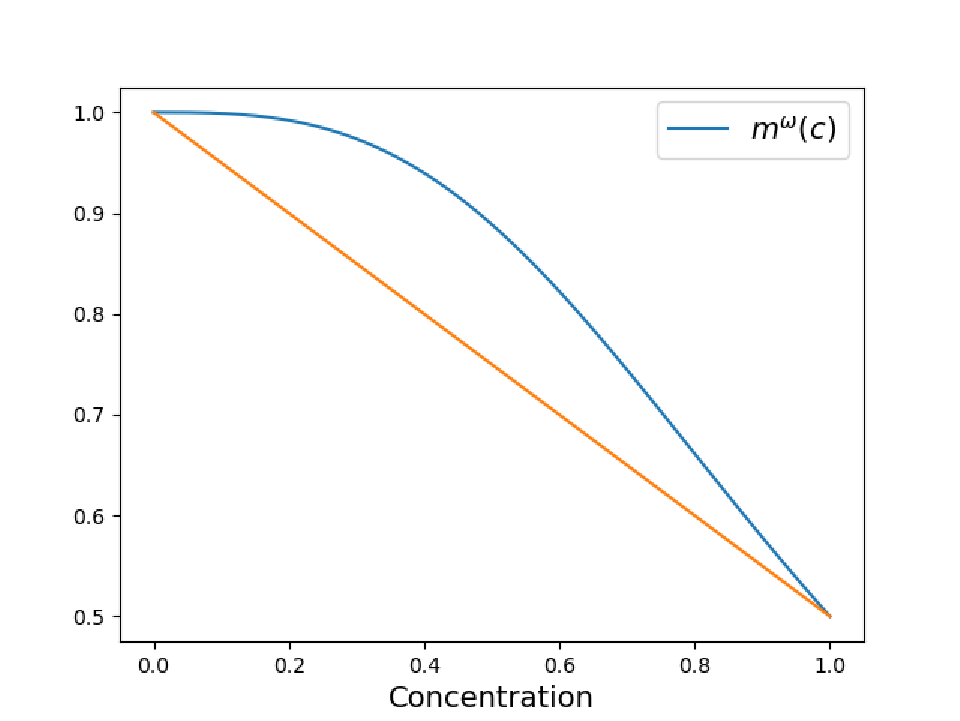}
    \includegraphics[scale = 0.36]{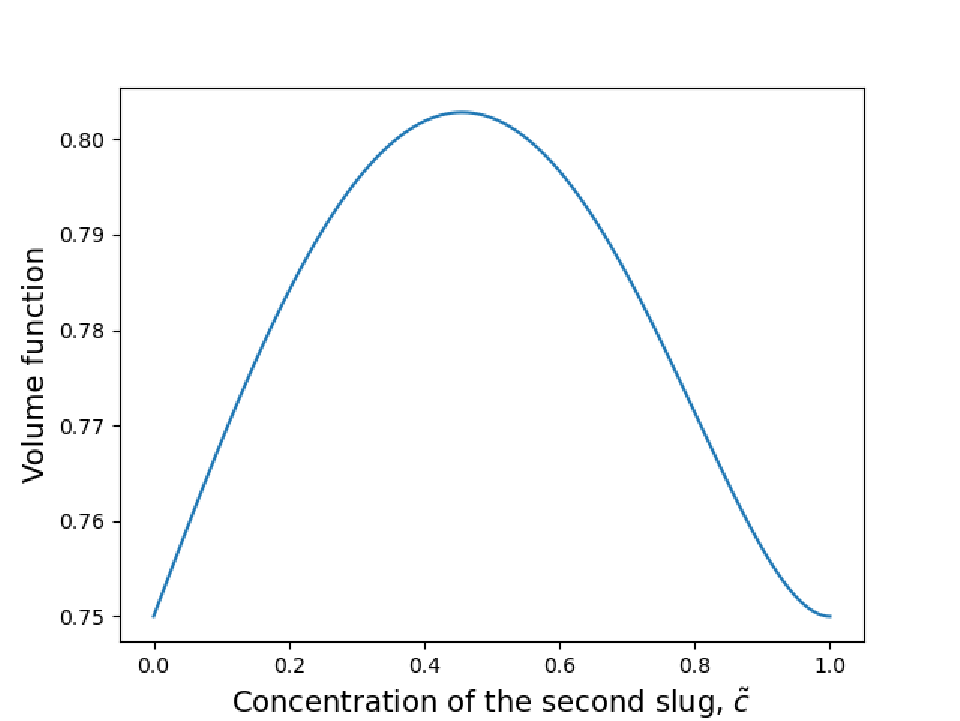}
    \caption{An example where GVB technology doesn't give positive results according to TL model. Left: graph of $m^\omega$; right: dependence of the volume of the polymer on the concentration of the second slug for the non-convex $m^\omega$, TL model}
    \label{tl_countrexample}
\end{figure}

\end{remark}

\begin{lemma}
\label{lm3}
    Let $(\bfc^n, \bft^n)$ be a sequence of optimal $n$-configurations. Then 
    $\delta(\bfc^n) \to 0$ as $n\to+\infty$.
\end{lemma}
\begin{proof}
        Proof by contradiction. Fix $\varepsilon$ and consider a subsequence $\{n_k\}$ such that 
        \begin{align*}
            \lvert c_{j_{n_k}}^{n_k} - c_{j_{n_k} + 1}^{n_k} \rvert \geq \varepsilon.
        \end{align*}
        Using Lemma \ref{lm2}, we obtain that there exists a function $\varphi=\varphi(\delta(\bfc^{n_k}))$, such that  for any configuration $\bfc^{n_k}$ there exists a configuration $\bfc^{n_k+1}$ with at least $\varphi(\delta(\bfc^{n_k}))$ gain in polymer volume, that is:
        \begin{align}
        \label{contradiction-1}
            V(\bfc^{n_k}, \bft^{n_k})>V(\bfc^{n_k+1},\bft^{n_k+1})+\varphi(\delta(\bfc^{n_k}))\geqslant V(\bfc^{n_k+1},\bft^{n_k+1})+\varphi(\varepsilon),
        \end{align}
        where the last inequality is true by monotonocity of $\varphi$.  Note that the sequence of $V(\bfc^{n_k},\bft^{n_k})$, $k\to\infty$, has a limit $V^{\infty}$ as non-increasing and non-negative. So starting from some $N_0$ for all $n_k>N_0$ we have 
        \begin{align}
        \label{contradiction-2}
            V^{\infty}+\frac{\varphi(\varepsilon)}{2}>V(\bfc^{n_k},\bft^{n_k})>V^{\infty}.
        \end{align}
        Combining \eqref{contradiction-1} and \eqref{contradiction-2} we arrive at a contradiction, since
$$        
            V^{\infty}+\frac{\varphi(\varepsilon)}{2}
            >
            V(\bfc^{n_k}, \bft^{n_k})
            >
            V(\bfc^{n_k+1},\bft^{n_k+1})+\varphi(\varepsilon)
            \geqslant
            V(\bfc^{n_{k+1}},\bft^{n_{k+1}})+\varphi(\varepsilon)
            >
            V^{\infty}+\varphi(\varepsilon).
$$
\end{proof}
Now we can prove Theorem~\ref{02-main_theorem}.
\begin{proof}
        Fix $c\in [c_{\min}, c_{\max}]$. For every $n$ choose $j_n$, such that $c_{j_n+1}^n \leq c < c_{j_n}^n$. Since $\delta(\bfc^n)\to0$ by Lemma~\ref{lm3}, we obtain
        \[
        c_{j_n}^n, c_{j_n+1}^n \to c, \qquad n\to\infty.
        \]
        
        Consider TFE model. Using~\eqref{Tj-TFE} we get
        \begin{gather*}
            1-\frac{m(c_1)}{m(c_{j_n}^n)} \leqslant \cT^n(c) \leqslant 1-\frac{m(c_1)}{m(c_{j_n+1}^n)},
        \end{gather*}
        thus $\cT^n\to \cT^\infty$. Here we use squeeze theorem in the following form: when a sequence lies between two other converging sequences with the same limit, it also converges to this limit \cite[section 2.28]{Fikh}.
        
        Consider generalized Koval model. Using formula~\eqref{Tj-Koval}, we get
           \begin{gather}
           \label{Koval-Tj-formula}
                    1-\cT^n(c) = f\left( \frac{m(c_{j_n+1}^n)}{m(c_{j_n}^n)} \right) \cdot\prod\limits_{k=1}^{j_n-1} f^2 \left( \frac{m(c_{k+1}^n)}{m(c_{k}^n)} \right).
            \end{gather}
         The key idea is the following. Due to Taylor formula and conditions (A) and (B) 
         $$
         \frac{f(1+\varepsilon)}{(1+\varepsilon)^{f'(1)}} -1 = o(\varepsilon) \quad \mbox{as \; $\varepsilon \to 0$}
         $$
       and hence
         \begin{align}
         \label{key-idea}
             f\left(\frac{m(x)}{m(y)}\right)=\left(\frac{m(x)}{m(y)}\right)^{f'(1)}(1+h(x, y)).
         \end{align}
        where the following estimate for the residue function $h$ holds:
        \[
        \widetilde{h}( d ) := \max_{|x-y| \leq d} \dfrac{|h(x, y)|}{|x-y|} \to 0 \qquad \text{as } d \to 0.
        \]
        Substituting \eqref{key-idea} into \eqref{Koval-Tj-formula}, we obtain
        \begin{equation}\label{eq:Koval_1-T}
                1-\cT^n(c) = \left(\frac{m(c_{j_n+1}^n)}{m(c_{j_n}^n)}\right)^{f'(1)} \prod\limits_{k=1}^{j_n-1} \left(\frac{m(c_{k+1}^n)}{m(c_{k}^n)}\right)^{2f'(1)}\cdot H^n_{j_n},
        \end{equation}
        \begin{equation*}
                H_{j_n}^n := (1 + h(c_{j_n+1}^n, c_{j_n}^n)) \prod_{k=1}^{j_n-1}(1+h(c_{k+1}^n,c_{k}^n))^2.
        \end{equation*}
        Consider $n\geq N$, such that $\widetilde{h}(\delta(\bfc^n)) (c_{\max} - c_{\min}) < \frac{1}{2}$. By the definition of function $\widetilde{h}$ we have
        \[
        \forall k=1, \ldots, j_n, \qquad h(c_{k+1}^n,c_{k}^n) \leq \widetilde{h}(\delta(\bfc^n)) (c_k^n - c_{k+1}^n).
        \]
        Thus, from the standard inequality for the logarithm
        \[
        \dfrac{x}{x+1} \leq \ln(1+x) \leq x 
        \]
        we obtain an estimate 
        \begin{equation}\label{eq:Koval_ln_H_estimate}
        -4\widetilde{h}(\delta(\bfc^n)) (c_{\max} - c_{\min}) \leq \ln H_{j_n}^n \leq 2 \widetilde{h}(\delta(\bfc^n)) (c_{\max} - c_{\min}), \qquad n\geq N.
        \end{equation}
        Combining \eqref{eq:Koval_1-T} and \eqref{eq:Koval_ln_H_estimate} for $n\geq N$ we arrive at
        \[
        \left(\frac{m(c^n_{j_n})m(c^n_{j_n+1})}{m(c^n_1)^2}\right)^{f'(1)} \gamma^{-4\widetilde{h}(\delta(\bfc^n))}    \leq 1 - \cT^n(c) \leq \left(\frac{m(c^n_{j})m(c^n_{j_n+1})}{m(c^n_1)^2}\right)^{f'(1)} \gamma^{2\widetilde{h}(\delta(\bfc^n))},
        \]
        where $\gamma = \exp (c_{\max} - c_{\min})$.
        Thus as $n\to\infty$ the statement follows immediately.
    \end{proof}


\section{Numerical solution for several slugs}\label{sec:numerics}

For a finite number of slugs, the optimization problem can be solved using numerical algorithms, for example, the conjugate gradient method. The results for $n$ slugs ($n=2,5,10,50$) and optimal limiting injection profiles are shown in Figs.~\ref{04-Otto_B_TL} and~\ref{04-Otto_B_TL_2}. 

Theorem \ref{02-main_theorem} for TFE and generalized Koval models answers the question of maximum potential polymer gain. Define the polymer gain for $n$ slugs in comparison to one slug as
\begin{gather*}
    \eta_n := \frac{V_1 - V_n}{V_1},
\end{gather*}
where $V_1$ is the value of the volume function \eqref{02-volume_function} for the optimal 1-configuration and $V_n$ is the value of the volume function \eqref{02-volume_function} for the optimal $n$-configuration. Quantities $V_1$ and $V_n$ have a simple analytic representation
\begin{align*}
    &V_1 = 1 - \frac{1}{v^f_1}, && V_n =\int_{c_{\min}}^{c_{\max}} \cT^n(c) \; dc.
\end{align*}
Dependence of the polymer volume gain for different numbers of slugs in comparison to one slug for TFE, Todd-Longstaff, naive Koval and Koval models are given in Tables \ref{03-polymer_gain_1} and \ref{03-polymer_gain_2}. 

It follows from Tables \ref{03-polymer_gain_1} and \ref{03-polymer_gain_2}, thet after $5$ slugs, polymer gain remains practically unchanged and is almost equal to the limit value. Hence, in practice it makes no sense to inject more than $5$ slugs. 
\begin{figure}
    \begin{center}
        \includegraphics[scale=0.30]{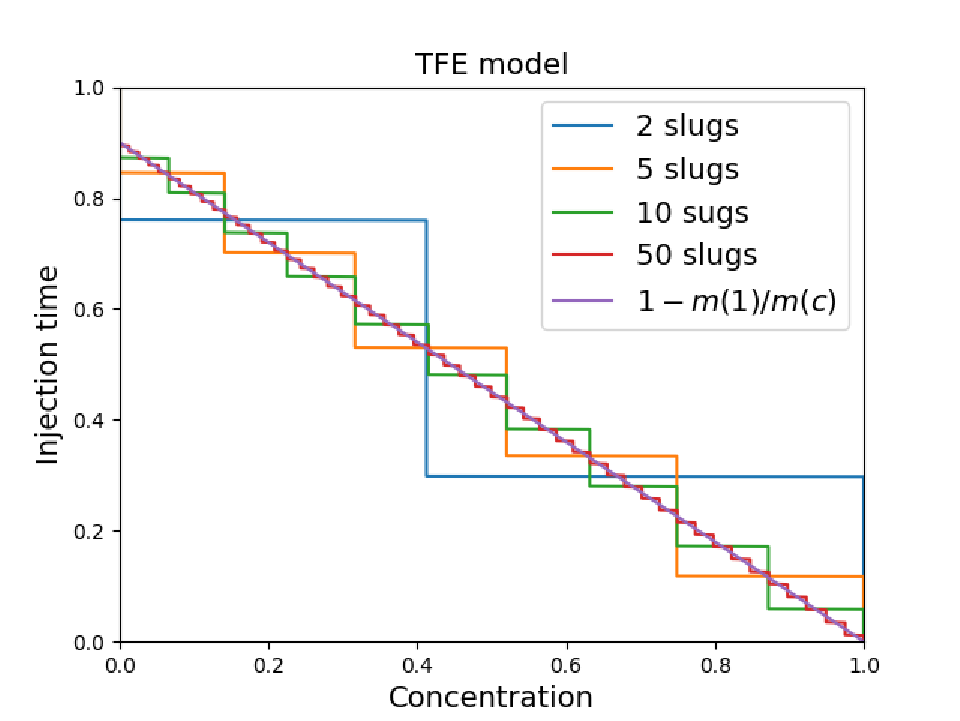}
        \includegraphics[scale=0.30]{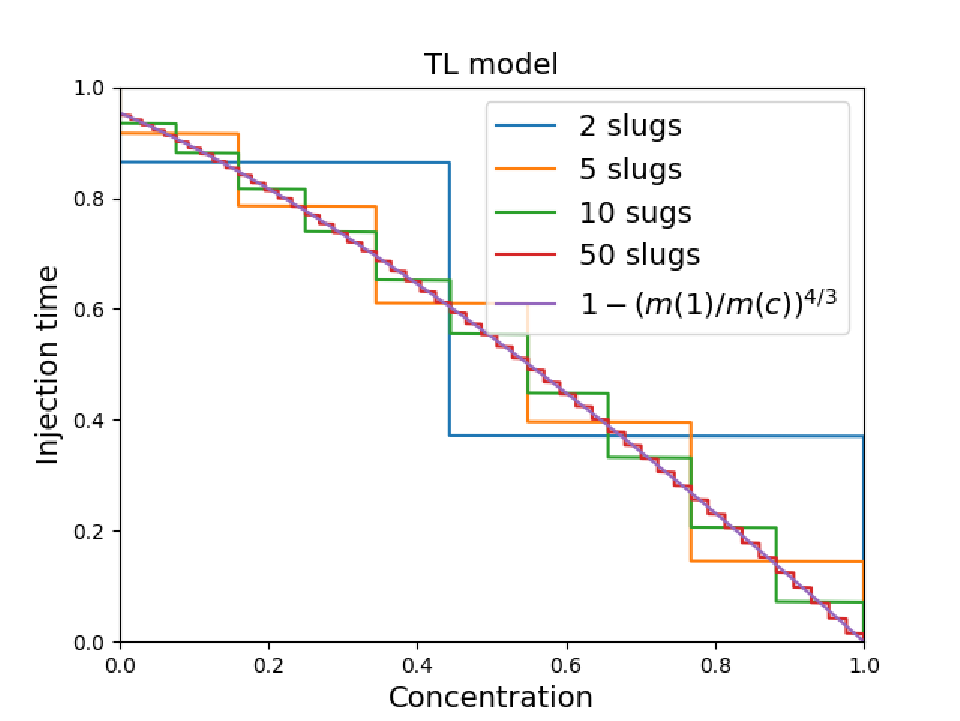}
        \\
        \includegraphics[scale=0.30]{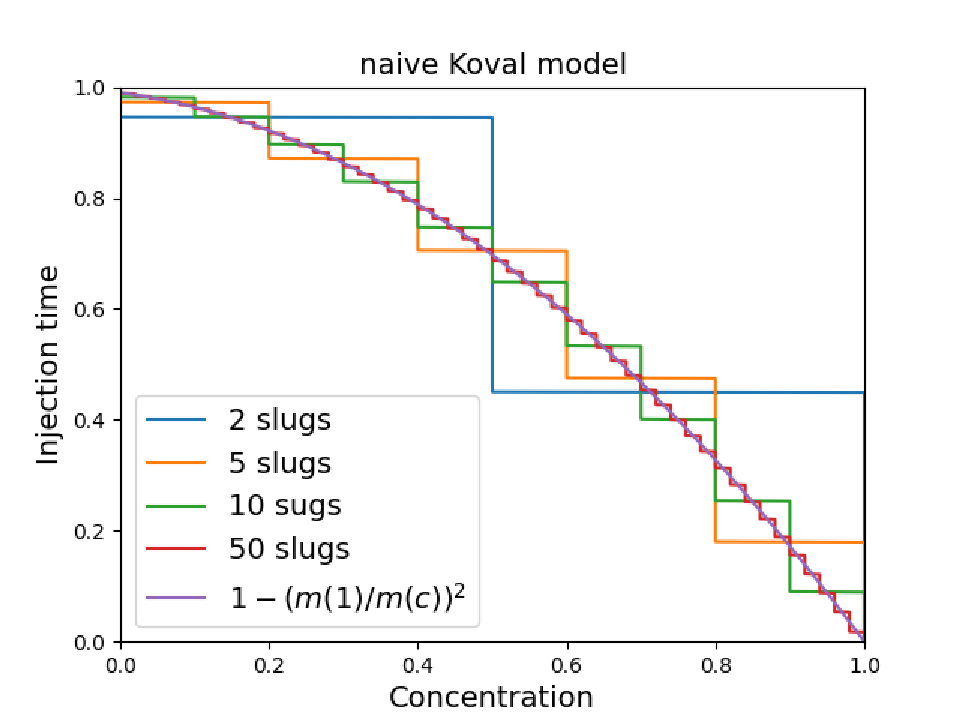}
        \includegraphics[scale=0.30]{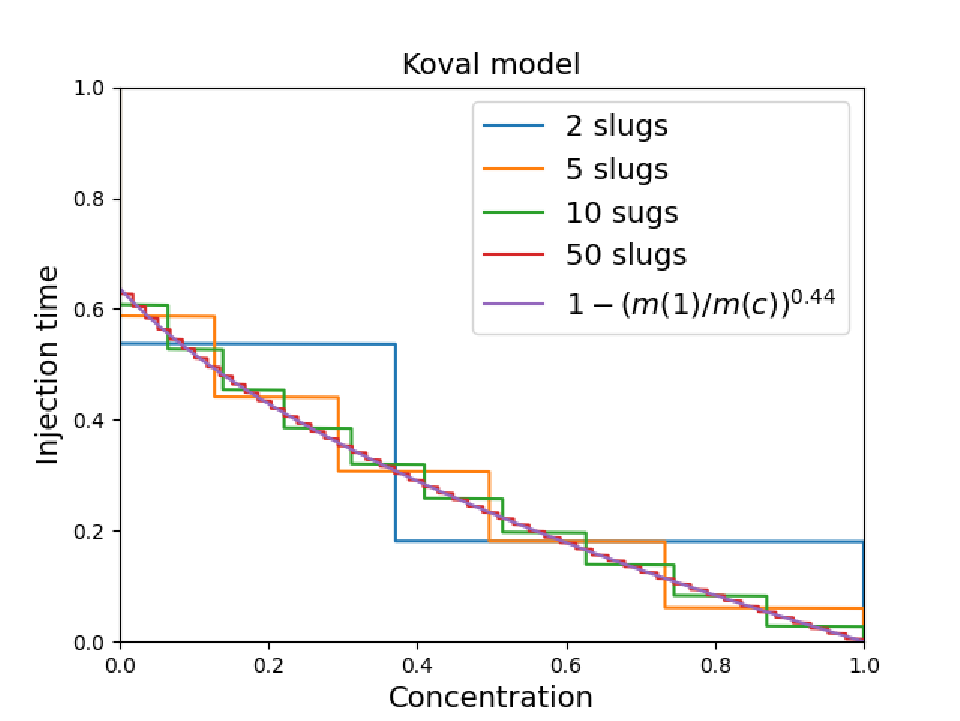}
    \end{center}
    \caption{Optimal injection profiles for $n$ slugs, $n=2,5,10,50$, and limiting optimal injection profiles for TFE, Todd-Longstaff, naive Koval and Koval models, $\mu(c) = 1 + 9c$}
    \label{04-Otto_B_TL}
\end{figure}
\begin{figure}
    \begin{center}
        \includegraphics[scale=0.30]{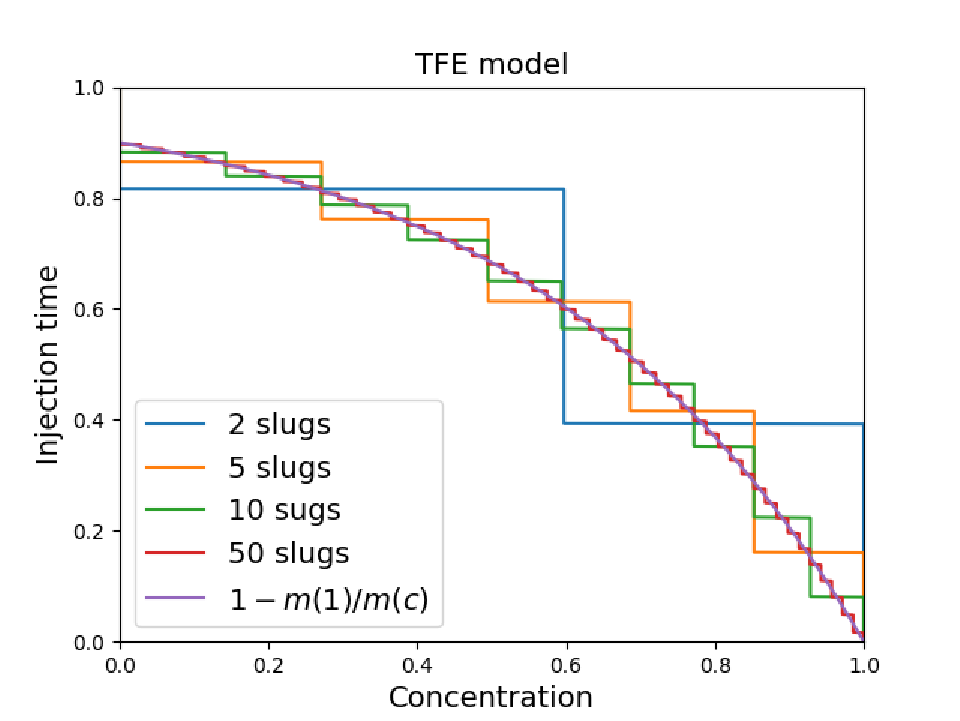}
        \includegraphics[scale=0.30]{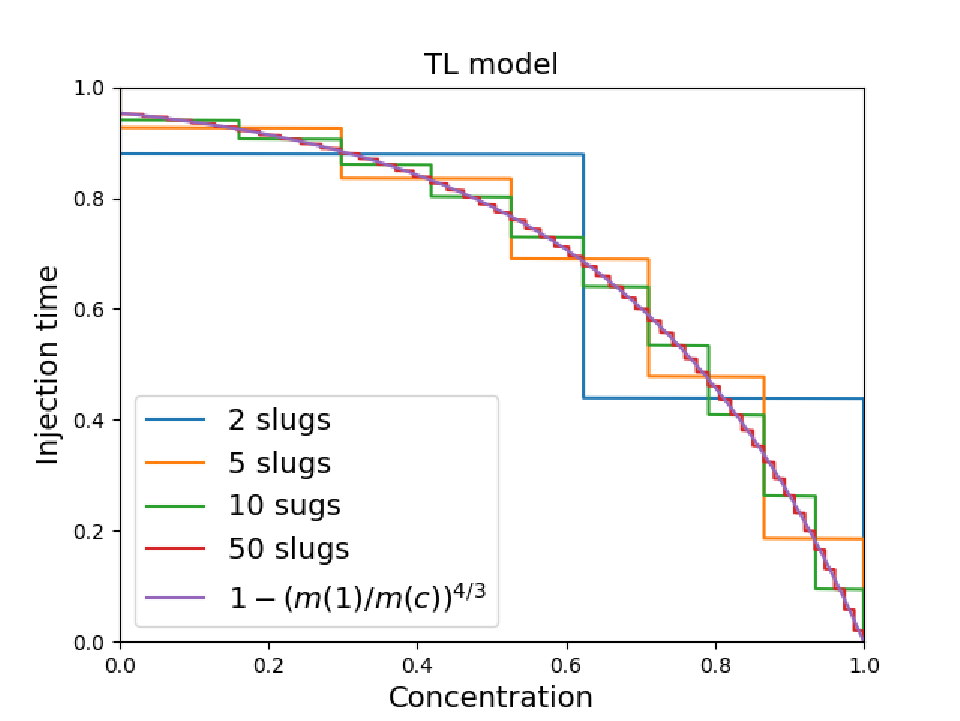}\\
        \includegraphics[scale=0.30]{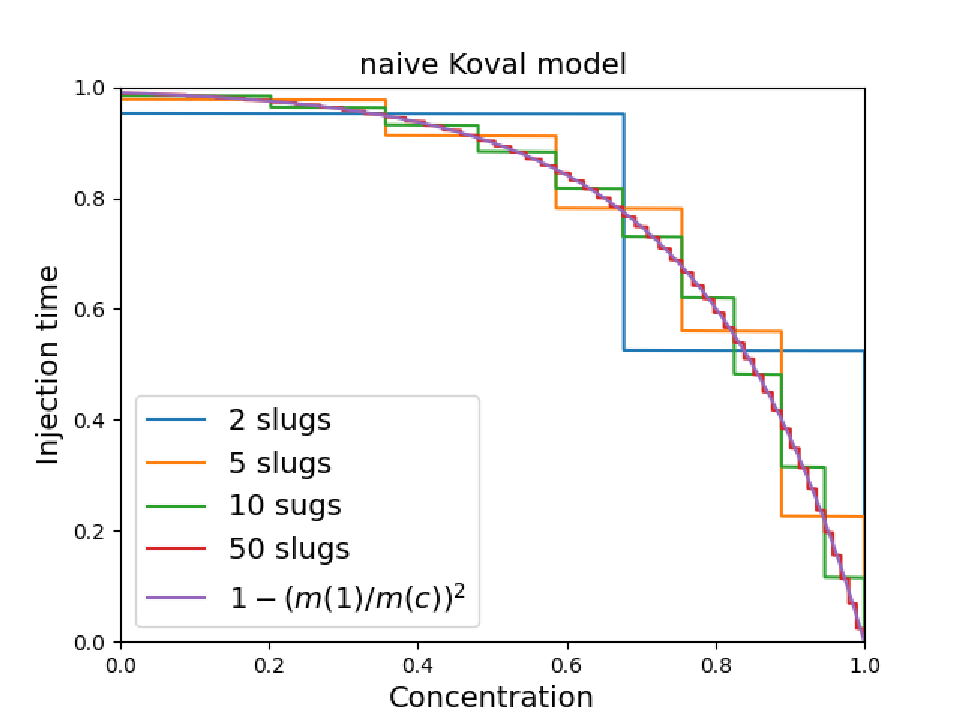}
        \includegraphics[scale=0.30]{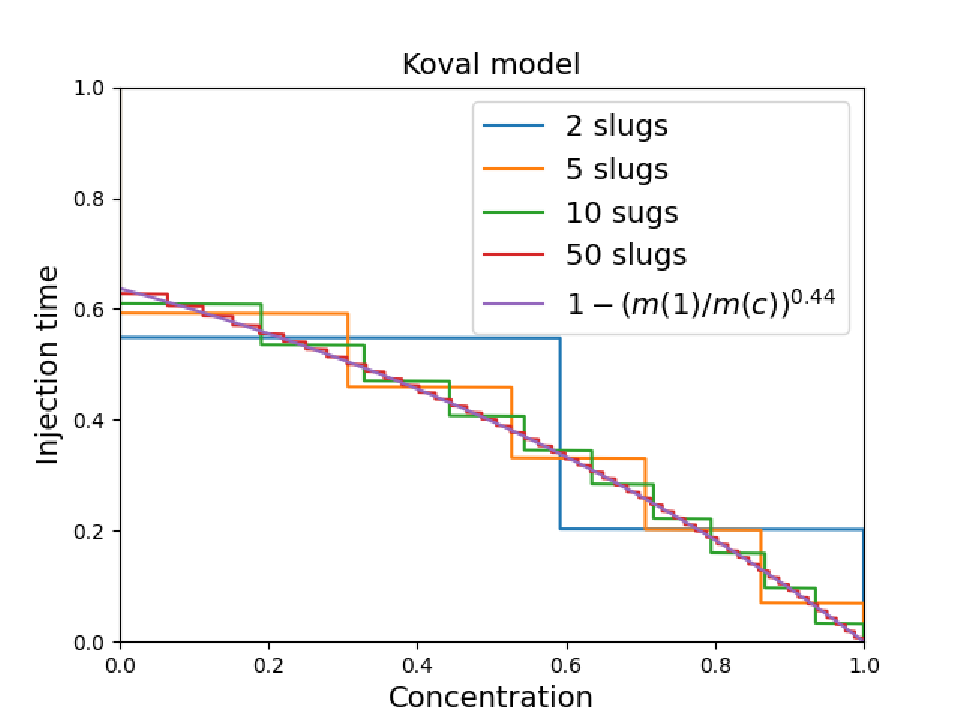}
    \end{center}
    \caption{Optimal injection profiles for $n$ slugs, $n=2,5,10,50$, and limiting optimal injection profiles for TFE, Todd-Longstaff, naive Koval and Koval models, $\mu(c) = \exp(c\ln10 )$}
    \label{04-Otto_B_TL_2}
\end{figure}
\begin{table}[!h]
    \centering
    \footnotesize
    \caption{The dependence of the polymer volume gain 
    on the number of slugs, $\mu(c) = 1 + 9c$ \label{03-polymer_gain_1}}
    \begin{tabular}{|c|c|c|c|c|c|c|}
        \hline
            & n = 2 & n = 3 & n = 4 & n = 5 & n = 10 & Limit \\
        \hline
        TFE & 19,83\% & 23,35\% & 24,57\% & 25,13\% & 25,88\% & 26,12\% \\
        Todd-Longstaff & 24,84\% & 29,36\% & 30,93\% & 31,66\% & 32,63\% & 32,95\% \\
        naive Koval & 22,50\% & 26,67\% & 28,12\% & 28,80\% & 29,70\% & 30,00\% \\
        Koval & 33,21\% & 39,24\% & 41,46\% & 42,55\% & 44,24\% & 45,28\%\\
        \hline
    \end{tabular}
\end{table}
\begin{table}[!h]
    \centering
    \footnotesize
    \caption{The dependence of the polymer volume gain 
    on the number of slugs, $\mu(c) = \exp(c\ln10)$ \label{03-polymer_gain_2}}    \begin{tabular}{|c|c|c|c|c|c|c|}
        \hline
            & n = 2 & n = 3 & n = 4 & n = 5 & n = 10 & Limit \\
        \hline
        TFE & 13,24\% & 15,93\% & 16,89\% & 17,34\% & 17,94\% & 18,14\% \\
        Todd-Longstaff & 9,07\% & 10,77\% & 11,36\% & 11,64\% & 12,01\% & 12,13\%\\
        naive Koval & 9,52\% & 11,32\% & 11,96\% & 12,25\% & 12,64\% & 12,78\%\\
        Koval & 13,06\% & 16,14\% & 17,47\% & 18,22\% & 19,57\% & 20,75\% \\
        \hline
    \end{tabular}
\end{table}

\section{Conclusion}
In the paper it is demonstrated that  graded viscosity banks technology allows to   significantly reduce total mass of the polymer slugs. Various models of mixing zone formation in miscible displacement, including most relevant in petroleum engineering: transverse flow equilibrium, Koval and Todd-Logstaff, are considered. For these models we present the optimal injection profile with rigorous mathematical proof and formulate the conditions for applicability of the technique. The results of numerical experiments confirm the optimal profile and suggest that the graded viscosity banks technology yields most of the benefits with the first 2-3 slugs.

\section*{Acknowledgements}

Theorem~\ref{02-main_theorem} and its proof in Section~\ref{sec:proof} was supported by the Russian Science Foundation grant 19-71-30002. The rest of the paper was supported by President Grant 075-15-2019-204, August Moebius Contest, and Ministery of Science and Education of Russian Federation, grant 075-15-2019-1619 and FAPERJ project ``Dynamics of approximations''.





\bibliography{sample}
%



\end{document}